\documentclass[10pt]{amsart}

\makeatletter
\def\blfootnote{\gdef\@thefnmark{}\@footnotetext}
\makeatother

\usepackage{epsfig}
\usepackage{graphics}
\usepackage{dcpic, pictexwd}
\usepackage{float}
\usepackage[colorlinks=true,
                    linkcolor=blue,
                    urlcolor=blue,
                    citecolor=blue,
                    anchorcolor=blue]{hyperref}
\usepackage{mathtools}
\usepackage{amsmath,amssymb}

\theoremstyle{plain}

\newtheorem*{theorem*}{Theorem}

\newtheorem{theorem}{Theorem}[section]
\newtheorem{lemma}[theorem]{Lemma}

\theoremstyle{remark}

\theoremstyle{Acknowledgments}

\newtheorem*{main}{Main Theorem}

\theoremstyle{definition}

\def\mod{{\rm Mod}}

\begin{document}
\blfootnote{\textup{2000} \textit{Mathematics Subject Classification}:
57M07, 20F05, 20F38}
\blfootnote{\textit{Keywords}:
Mapping class groups, punctured surfaces, involutions, generating sets}
\newenvironment{prooff}{\medskip \par \noindent {\it Proof}\ }{\hfill
$\square$ \medskip \par}
    \def\sqr#1#2{{\vcenter{\hrule height.#2pt
        \hbox{\vrule width.#2pt height#1pt \kern#1pt
            \vrule width.#2pt}\hrule height.#2pt}}}
    \def\square{\mathchoice\sqr67\sqr67\sqr{2.1}6\sqr{1.5}6}
\def\pf#1{\medskip \par \noindent {\it #1.}\ }
\def\endpf{\hfill $\square$ \medskip \par}
\def\demo#1{\medskip \par \noindent {\it #1.}\ }
\def\enddemo{\medskip \par}
\def\qed{~\hfill$\square$}

 \title[ Generators of $\mod(N_{g,p})$] {Generators of the Mapping Class Group of a Nonorientable Punctured Surface}

\author[T{\"{u}}l\.{i}n Altun{\"{o}}z,       Mehmetc\.{i}k Pamuk, and O\u{g}uz Y{\i}ld{\i}z ]{T{\"{u}}l\.{i}n Altun{\"{o}}z,    Mehmetc\.{i}k Pamuk, and O\u{g}uz Y{\i}ld{\i}z}
\address{Faculty of Engineering, Ba\c{s}kent University, Ankara, Turkey} 
\email{tulinaltunoz@baskent.edu.tr} 
\address{Department of Mathematics, Middle East Technical University,
 Ankara, Turkey}
 \email{mpamuk@metu.edu.tr}
 \address{Department of Mathematics, Middle East Technical University,
 Ankara, Turkey}
  \email{oguzyildiz16@gmail.com}

%\subjclass{57M99, 20F38}
%\date{\today}
%\keywords{Mapping class groups, Lefschetz fibrations, Bounded cohomology}
%\thanks{* Supported by T\"UBA/GEB\.IP}

\begin{abstract}   Let $\mod(N_{g, p})$ denote the mapping class group of a  nonorientable surface of genus $g$ with $p$ punctures. 
For $g\geq14$, we show that $\mod(N_{g, p})$ can be generated by five elements or by six involutions.  
\end{abstract}
\maketitle
%\tableofcontents
  \setcounter{secnumdepth}{2}
 \setcounter{section}{0}
 
\section{Introduction}
Let $N_{g,p}$ denote a nonorientable surface of genus $g$ with $p$ punctures specified by the set $P=\lbrace z_1, z_2, \ldots, z_p\rbrace$ of $p$ distinguished points.
If $p$ is zero,  we omit it from the notation and denote the surface by $N_{g}$.   {\textit{The mapping class group}} 
$\mod(N_{g, p})$ of the surface $N_{g, p}$ is defined to be the group of the isotopy classes of self-diffeomorphisms of $N_{g, p}$ which preserve the set $P$.   
In this paper, we consider only connected surfaces and we are interested in finding minimal generating sets for $\mod(N_{g, p})$.

For an orientable surface, the mapping class group is defined as the group of isotopy classes of orientation-preserving diffeomorphisms. 
%One includes all diffeo- morphisms in the de¢nition of the extended mapping class group. The pure mapping class group of F is the subgroup of the mapping class group 
%consisting of those mapping classes which ¢x each puncture. The pure mapping class group of F is generated by Dehn twists. 
For orientable surfaces, it is known that the mapping class group is generated by Dehn twists and elementary braids. %, and the extended mapping class group is generated by Dehn twists, elementary braids and the isotopy class of an orientation-reversing diffeomorphism. The number of generators in each case can be chosen to be ¢nite.
However, if the surface  is nonorientable,  Dehn twists do not generate the mapping class group.   
Lickorish was the first to consider the problem of finding generators of the group $\mod(N_{g})$. He proved that $\mod(N_{g})$ is generated by Dehn twists 
and the isotopy class of a homeomorphism he called the $Y$-homeomorphism~\cite{lickorish}. Chillingworth~\cite{c} determined a finite set of generators for $\mod(N_{g})$. 
The cardinality of this set depends on $g$. Korkmaz~\cite{mk4} extended Chillingworth’s result to the case of punctured surfaces.  

Since $\mod(N_g)$  is not abelian, a generating set must contain at least two elements. Szepietowski~ \cite{sz2}  proved that this group is generated by three elements for all $g \geq 3$.
We~\cite{apy2} prove that, for $g \geq 19$ the mapping class group of a nonorientable surface of genus $g$, $\mod(N_{g})$, can be generated by two elements, one of which is of order $g$.  
 
In the presence of punctures, Szepietowski~\cite[Theorem~3]{sz1}  proved that $\mod(N_{g, p})$ is generated by involutions.   We prove that for $g \geq 26$, $\mod(N_{g})$  can be generated 
by three involutions if $g \geq 26$ ~\cite{apy2}.  Recently, Yoshihara~\cite{yoshihara}  proved that, for $p \geq0$,  $\mod(N_{g, p})$ is generated by eight involutions if $g \geq 13$ is odd and 
by eleven involutions if $g \geq 14$ is even.

%        Let $\Sigma_{g,p}$ be a connected orientable surface of genus-$g$ with $p$-punctures (when $p=0$ we write  $\Sigma_{g}$). 
%        The \textit{mapping class group} of $\Sigma_{g,p}$ is the group of isotopy classes of orientation-preserving homeomorphisms of $\Sigma_{g,p}$ preserving the set of punctures.
%          Note that, since  $\mod(\Sigma_{g,p})$ contains nonabelian free groups, it cannot be generated  by two involutions. 

In this paper, we obtain the following main result (see Theorems \ref{element} and \ref{invo}):
\begin{main}\label{main}
For $g\geq 14$ and $p\geq 1$, $\mod(N_{g,p})$ is generated by five elements or six involutions.
\end{main}

The paper is organized as follows: In Section~\ref{S2}, we present the necessary background and some results on mapping class groups. 
The proof of our main theorem is given in Section~\ref{S3}.

\medskip

\noindent
{\bf Acknowledgements.}
This work is supported by the Scientific and Technological Research Council of Turkey (T\"{U}B\.{I}TAK)[grant number 120F118].

%%%%%%%%%%%%%%%%%%%%%%%%%%%%%%%%%%%%%%%%%%%%%%%%%%%%%%%%%%%%%%%%%%%%%%%%%%%%%%%%%%%%%%%%%%%%%%%%

%\section{ preliminaries}\label{S2}
\par  
\section{Background and Results on Mapping Class Groups} \label{S2}

 %Let $N_{g,p}$ denote a closed nonorientable surface of genus $g$ with $p$ punctures specified by the set $P=\lbrace z_1,z_2,\ldots,z_p\rbrace$ of $p$ distinguished points. If $p$ is zero then we omit it from the notation and write $N_{g}$. {\textit{The mapping class group}} 
% $\mod(N_{g,p})$ of the surface $N_{g,p}$ is defined to be the group of the isotopy classes of 
 %self-diffeomorphisms of $\Sigma_{g,p}$ which fix the set $P$. 
 
 Let $\mod_{0}(N_{g,p})$ denote the subgroup of $\mod(N_{g,p})$ consisting of elements which fix the set $P$ pointwise. 
This group is a normal subgroup of $\mod(N_{g,p})$ of index $p!$ and  we have the following exact sequence:
 \[
1\longrightarrow \mod_{0}(N_{g,p})\longrightarrow \mod(N_{g,p}) \longrightarrow Sym_{p}\longrightarrow 1,
\]
where $Sym_p$ is the symmetric group on the set $\lbrace1,2,\ldots,p\rbrace$ and the last projection is given by the restriction of the isotopy class of a diffeomorphism to its action on the punctures. \par
 Throughout the paper we do not distinguish a 
 diffeomorphism from its isotopy class. For the composition of two diffeomorphisms, we
use the functional notation; if $f$ and $g$ are two diffeomorphisms, then
the composition $fg$ means that $g$ is applied first and then $f$.\\
\indent
 For a two-sided simple closed 
curve $a$ on $N_{g,p}$, we denote the right-handed 
Dehn twist $t_a$ about $a$ by the corresponding capital letter $A$.

Now, let us remind the following basic facts of Dehn twists that we use frequently in the remaining of the paper. Let $a$ and $b$ be two-sided
simple closed curves on $N_{g,p}$ and $f\in \mod(N_{g,p})$.
\begin{itemize}
\item  If $a$ and $b$ are disjoint, then $AB=BA$ (\textit{Commutativity}).
\item If $f(a)=b$, then $fAf^{-1}=B^{\varepsilon}$, where $\varepsilon=\pm 1$ depends on the orientation of a regular neighbourhood of $f(a)$ with respect to the chosen orientation.  (\textit{Conjugation}).
\end{itemize}
Throughout the paper we use repeatedly the conjugation relation and we denote $fgf^{-1}$ by $g^{f}$ for any $f,g \in \mod(N_{g,p})$.\\
\noindent
We now define two types of diffeomorphisms of $\mod(N_{g,p})$; $y$ homeomorphism (crosscap slide) or puncture slide.
\begin{figure}[hbt!]
\begin{center}
\scalebox{0.3}{\includegraphics{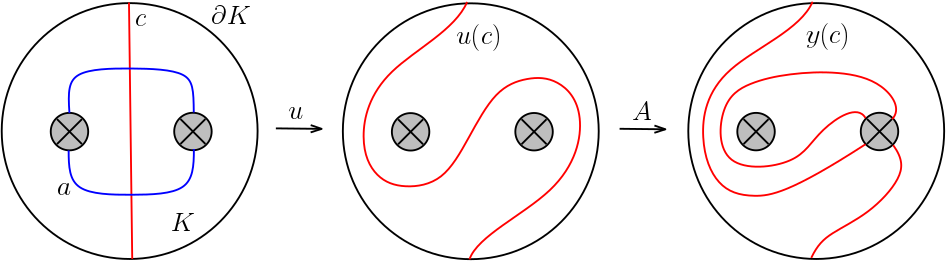}}
\caption{The homeomorphisms $u$ and $y = Au$.}
\label{CS}
\end{center}
\end{figure}
Let us consider the Klein bottle $K$ depicted in Figure~\ref{CS}. A \textit{crosscap transposition} $u$ is defined to be the isotopy classes of a diffeomorphism that interchanges two consecutive crosscaps and keeps the boundary of $K$ fixed as shown on the left hand side of Figure~\ref{CS}. The diffeomorphism $y = Au$ whose effect on the interval $c$ as in Figure~\ref{CS} can be also obtained as sliding a M{\"{o}}bius band once along the core of another one and fixing each point of the boundary of $K$. This is called a \textit{ $Y$-homeomorphism}~\cite{lickorish} (also called a \textit{crosscap slide}~\cite{mk4}).  Let us note that $A^{-1}u$ is a $Y$-homemorphism i.e. the other choice of the orientation for a neighbourhood of the curve $a$ also gives a $Y$-homeomorphism. Also note that $y^{2}$ is a Dehn twist about $\partial K$.

Lickorish~\cite{lickorish} proved that $\mod(N_g)$ is generated by Dehn twists and a $Y$ homeomorphism ( one crosscap slide). Note that crosscap transpositions can be used instead of crosscap slides in generating sets of $\mod(N_g)$ since a crosscap transposition is equal to the product of a Dehn twist and a crosscap slide.

For the other type of diffeomorphisms of $\mod(N_{g,p})$, consider the M{\"{o}}bius band $M$ with a puncture $z$ and the one-sided simple closed curve $a$ in Figure~\ref{PS1}.
\begin{figure}[hbt!]
\begin{center}
\scalebox{0.3}{\includegraphics{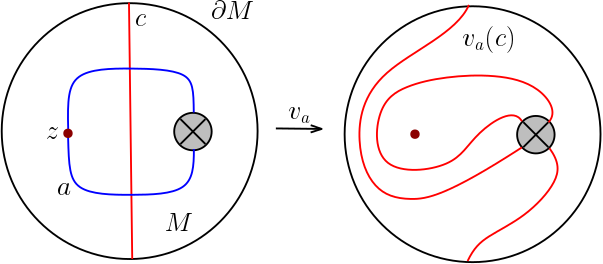}}
\caption{The puncture slide $v_a$.}
\label{PS1}
\end{center}
\end{figure}
We denote $v_a$ the \textit{puncture slide} obtained by sliding the puncture $z$ once along $a$ keeping each point of the boundary $M$ fixed. The effect of the diffeomorphism $v_a$ on the interval $c$ is shown as Figure~\ref{PS1}. Let us denote the puncture slide supported by the one-sided simple closed curve $\delta_i$ and the puncture $z_j$ by $v_{i,j}$ as shown in Figure~\ref{PS1}.
\begin{figure}[H]
\begin{center}
\scalebox{0.25}{\includegraphics{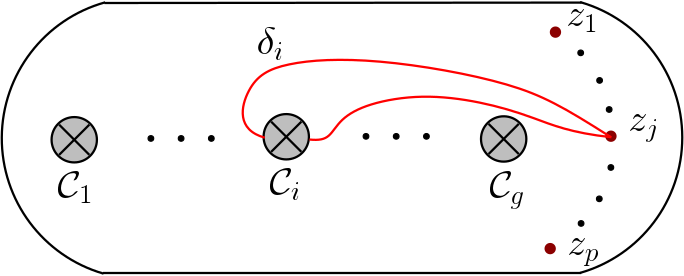}}
\caption{The puncture slide $v_{i,j}$.}
\label{PS2}
\end{center}
\end{figure}
Before we finish Preliminaries, let us give Korkmaz's generating set for $\mod_{0}(N_{g,p})$. 

Let $\mathcal{C}$ be the set of simple closed curves given by
\begin{eqnarray*}
\mathcal{C}&=&\lbrace a_i, b_i, f_i, c_j, e_k, \textrm{ for } i=1,\ldots,r, \quad j=1,\ldots,r-1 \textrm{ and }  k=1,\ldots, p-1 \textrm{ if } g=2r+1,\\
\textrm{ and }&\\ 
\mathcal{C}&=&\lbrace a_i, b_i, f_i, c_i, e_k, \textrm{ for } i=1,\ldots,r \textrm{ and }  k=1,\ldots, p-1 \textrm{ if } g=2r+2,
\end{eqnarray*}
where the curves are shown in Figure~\ref{G}.
\begin{figure}[hbt!]
\begin{center}
\scalebox{0.57}{\includegraphics{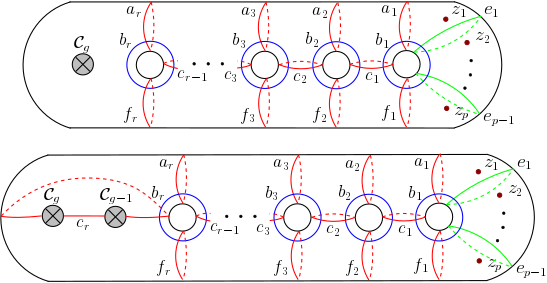}}
\caption{The Dehn twist generators for $\mod(N_{g,p})$.}
\label{G}
\end{center}
\end{figure}
\begin{theorem}\cite{mk4}\label{thmmk}
For $g\geq3$, the mapping class group $\mod_{0}(N_{g,p})$ is generated by
\begin{itemize}
\item[(i)] $\lbrace t_c, y, v_{g,i} :c \in \mathcal{C}, 1\leq i \leq p  \rbrace$ if $g$ is odd, and
\item[(ii)] $\lbrace t_c, y, v_{g-1,i}, v_{g,i} :c \in \mathcal{C}, 1\leq i \leq p  \rbrace$ if $g$ is even.
\end{itemize}
\end{theorem}

\section{A generating set for $\mod(N_{g,p})$}\label{S3}
We, first  recall the following basic lemma from algebra.
\begin{lemma}\label{lemma1}
Let $G$ and $K$ be groups. Suppose that the following short exact sequence holds,
\[
1 \longrightarrow N \overset{i}{\longrightarrow}G \overset{\pi}{\longrightarrow} K\longrightarrow 1.
\]
Then the subgroup $\Gamma$ contains $i(N)$ and has a surjection to $K$ if and only if $\Gamma=G$.
\end{lemma}
\par

For $G=\mod(N_{g,p})$ and $N=\mod_{0}(N_{g,p})$, %(self-diffeomorphisms fixing the punctures pointwise),
 the following short exact sequence holds:
\[
1\longrightarrow \mod_{0}(N_{g,p})\longrightarrow \mod(N_{g,p}) \longrightarrow Sym_{p}\longrightarrow 1.
\]
%where $S_p$ denotes the symmetric group on the set $\lbrace1,2,\ldots,p\rbrace$. 
Therefore, the following useful result  follows immediately from Lemma~\ref{lemma1}. Let $\Gamma$ be a subgroup of $\mod(N_{g,p})$. If the subgroup $\Gamma$ contains $\mod_{0}(N_{g,p})$ and has a surjection to $Sym_p$ then $\Gamma=\mod(N_{g,p})$.

\begin{figure}[hbt!]
\begin{center}
\scalebox{0.3}{\includegraphics{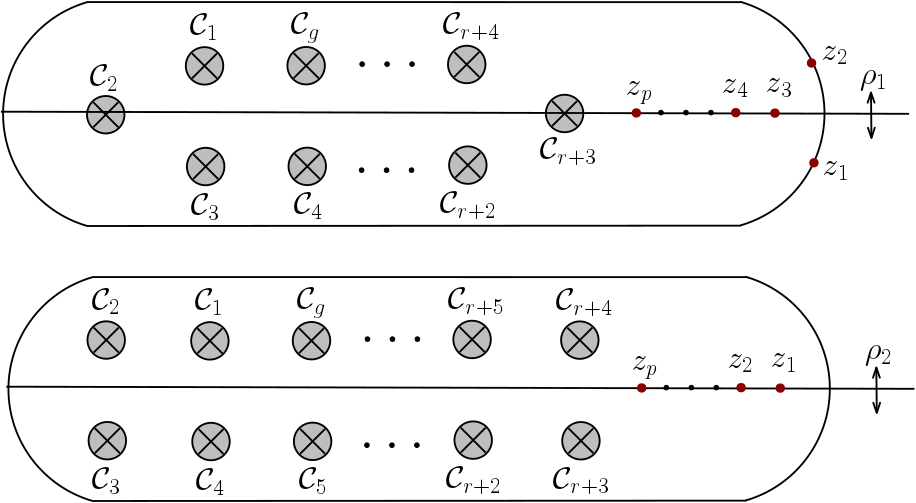}}
\caption{The reflections $\rho_1$ and $\rho_2$ on $N_{g,p}$ if $g=2r+2$.}
\label{ER12}
\end{center}
\end{figure}

\begin{figure}[hbt!]
\begin{center}
\scalebox{0.23}{\includegraphics{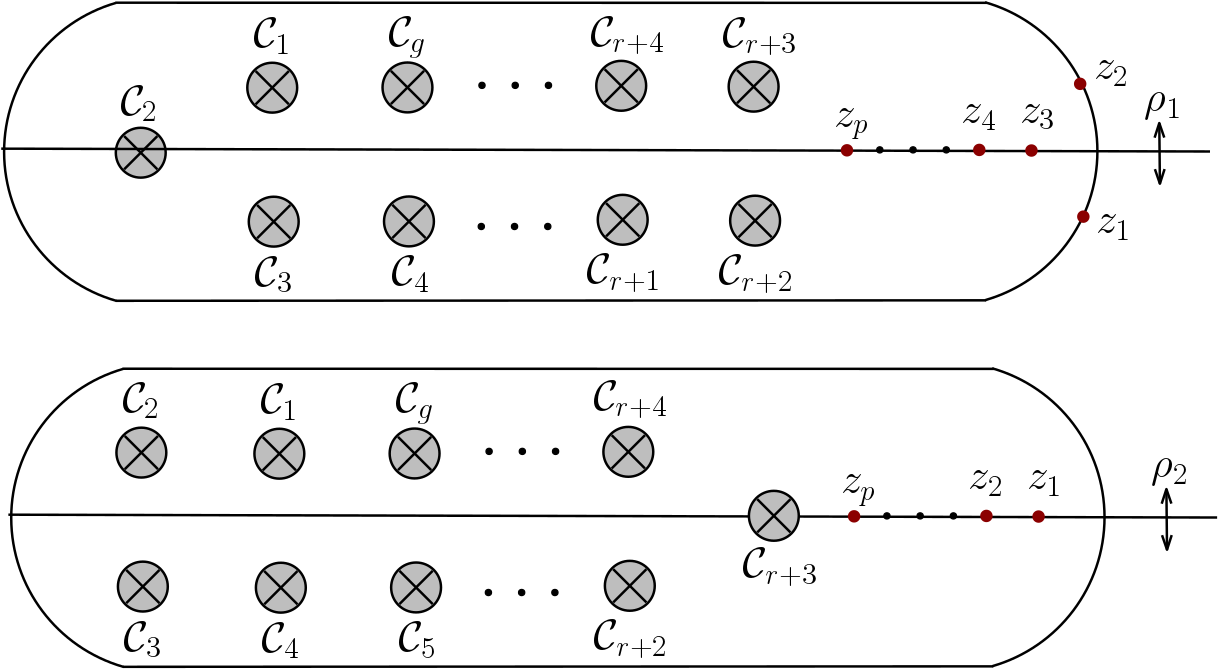}}
\caption{The reflections $\rho_1$ and $\rho_2$ on $N_{g,p}$ if $g=2r+1$.}
\label{OR12}
\end{center}
\end{figure}

\begin{figure}[hbt!]
\begin{center}
\scalebox{0.3}{\includegraphics{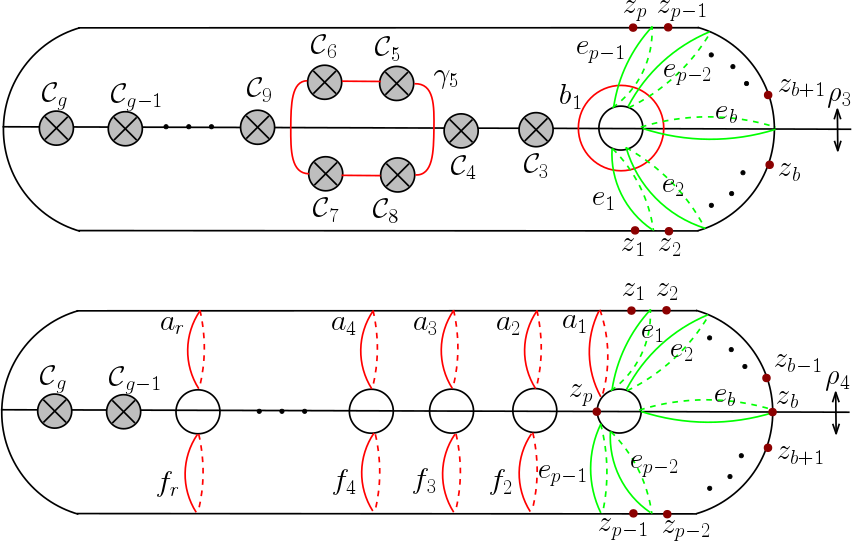}}
\caption{The reflections $\rho_3$ and $\rho_4$ on $N_{g,p}$ if $p=2b$  (when $g$ is odd we remove the last crosscap).}
\label{ER34}
\end{center}
\end{figure}

\begin{figure}[hbt!]
\begin{center}
\scalebox{0.27}{\includegraphics{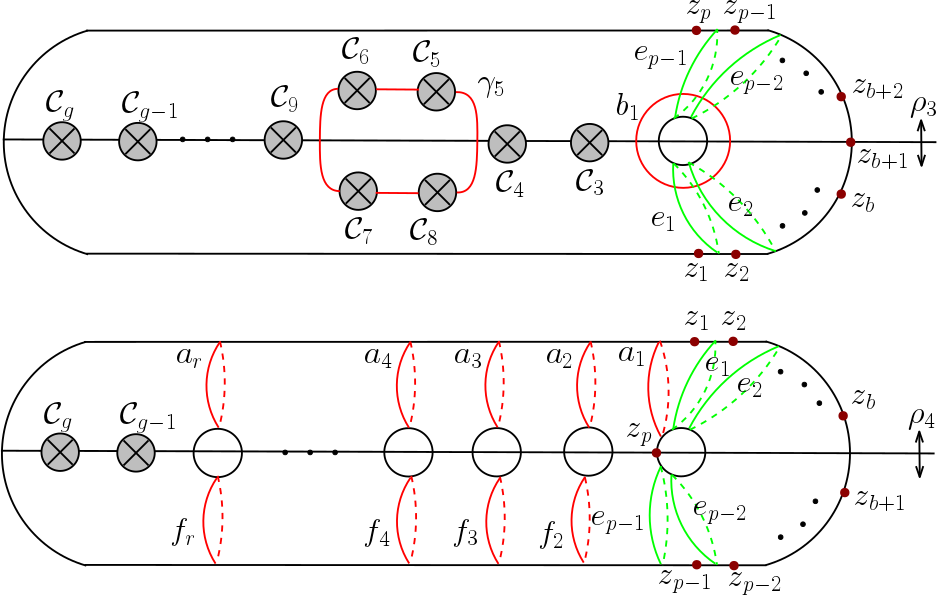}}
\caption{The reflections $\rho_3$ and $\rho_4$ on $N_{g,p}$ if $p=2b+1$  (when $g$ is odd we remove the last crosscap).}
\label{OR34}
\end{center}
\end{figure}
Some Dehn twists curves given in Theorem~\ref{thmmk} are shown in Figure~\ref{curves} when we use the model for the surface $N_{g,p}$ as a sphere with $p$ punctures and $g$ crosscaps.
\begin{figure}[hbt!]
\begin{center}
\scalebox{0.36}{\includegraphics{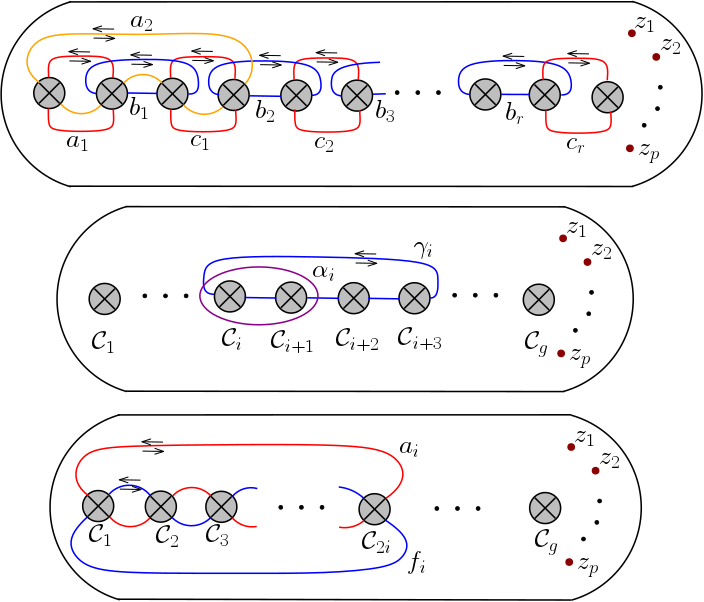}}
\caption{The curves $a_i,b_i,c_i,\gamma_i,f_i$ and $\alpha_i$ (note that we do not have $c_r$ when $g$ is odd).}
\label{curves}
\end{center}
\end{figure}

Consider the models for $N_{g,p}$ such that it is invariant under the reflections $\rho_1$ and $\rho_2$ (see Figures~\ref{ER12} and~\ref{OR12}). Note that $\mod(N_{g,p})$ contains the element $T=\rho_2\rho_1$ which satisfies the following:
\begin{itemize}
\item $T(b_i)=c_{i}$ and $T(c_j)=b_{j+1}$ for $i,j=1,\ldots, r-1$,
\item $T(a_1)=b_{1}$.  Also, note that $T^2(b_r)=a_1$ if $g=2r+1$ and $T(b_r)=c_r$, $T^2(c_r)=a_1$ if $g=2r+2$,
\item $T(z_1)=z_2$, $T(z_2)=z_1$ and $T(z_k)=z_k$ for $k=3, 4, \ldots, p$,
\item $T(\alpha_l)=\alpha_{l+1}$  for $l=1, 2, \ldots, g-1$ and $T(\alpha_g)=\alpha_1$
where all curves are depicted in Figure~\ref{curves}.
\end{itemize}
We will use two different diffeomorphism of $\mod(N_{g,p})$, the reflections $\rho_3$ and $\rho_4$ depicted in Figures~\ref{ER34} and ~\ref{OR34}, in the remaining part of the paper.

\begin{lemma}\label{lemma2}
For any  $g=2r+1\geq 15$ or $g=2r+2\geq 14$, the group generated by the elements 
\[
\lbrace T, \rho_3, \rho_4, \Gamma_5B_1u_{r+5}, v_{r+3,p} \rbrace
\]
contains the Dehn twists $A_i$, $B_i$, $F_i$ and $C_j$ for $i=1,\ldots,r$ and $j=1,\ldots,r-1$ (also the element $C_r$ when $g$ is even).
\end{lemma}
\begin{proof}
 Let $G_1:=\Gamma_5B_1u_{r+5}$ and let $G$ be the subgroup of $\mod(N_{g,p})$ generated by the elements $T$, $\rho_3$, $\rho_4$, $G_1$ and $v_{r+3, p}$.

 Let $G_2$ be the element obtained by the conjugation of $G_1$ by $T^{2}$. It follows from
\[
T^{2}(\gamma_{5},b_{1},\alpha_{r+5})=(\gamma_{7},b_{2},\alpha_{r+7})
\]
that
\[
G_2=G_{1}^{T^{2}}=\Gamma_7 B_2 u_{r+7}\in G. 
\]
 
 Let $G_3$ denote the element $G_1^{G_1G_2}$, which is contained in the subgroup $G$. Therefore,
\[
G_3=B_2B_1u_{r+5}
\]
Let us give more details of this calculation  since we use similar calculations in the rest of the proof.  It can be verified that the diffeomorphism $G_1G_2$ sends the curves $\lbrace \gamma_5, b_1, \alpha_{r+5} \rbrace$ to the curves $\lbrace b_2, b_1, \alpha_{r+5} \rbrace$, respectively. Then we get the element
\begin{eqnarray*}
G_3&=&G_1^{G_1G_2}\\
&=&(G_1G_2)(\Gamma_5 B_1 u_{r+5})(G_1G_2)^{-1}\\
&=&B_2 B_1u_{r+5}.
\end{eqnarray*}
 Thus the subgroup $G$ contains the  element
 \[
 G_1G_3^{-1}=\Gamma_5B_2^{-1}.
 \]
This implies that the element
\[
B_3\Gamma_7^{-1}=(B_2\Gamma_5^{-1})^{T^2}
\] 
 is also contained in $G$. Using these elements, we also obtain 
\begin{eqnarray*}
G_4&=&(B_3\Gamma_7^{-1})G_2=B_3B_2u_{r+7}\in G,\\
G_5&=&G_{4}^{T^{-3}}=C_1A_1u_{r+4}\in G,\\
G_6&=&G_4^{G_4G_5}=B_3C_1u_{r+7} \in G.
\end{eqnarray*} 
From these, we get the element 
 $G_4G_6^{-1}=B_{2}C_{1}^{-1}$, which is in $G$.
% \noindent
Hence, we have the following elements:
 \begin{eqnarray*}
C_1B_1^{-1}&=&(B_2C_1^{-1})^{T^{-1}},\\
B_1B_2^{-1}&=&(B_1C_1^{-1})(C_1B_2^{-1}),\\
C_1C_2^{-1}&=&(B_1B_2^{-1})^{T},\\
B_2C_2^{-1}&=&(B_1C_1^{-1})^{T^{2}},\\
\Gamma_5C_2^{-1}&=&(\Gamma_5B_{2}^{-1})(B_2C_2^{-1})\textrm{ and }\\
A_2A_1^{-1}&=&\Gamma_1A_1^{-1}=(\Gamma_5C_2^{-1})^{T^{-4}},
\end{eqnarray*}
which are all contained in $G$.
%\noindent
Furthermore, the subgroup $G$ contains the following elements:
\begin{eqnarray*}
A_1B_1^{-1}&=&(C_1B_2^{-1})^{T^{-2}},\\
A_1B_2^{-1}&=&(A_1B_1^{-1})(B_1B_2^{-1}),\\
A_1C_1^{-1}&=&(A_1B_1^{-1})(B_1B_2^{-1})(B_2C_1^{-1}),\\
C_1A_2^{-1}&=&(C_1A_1^{-1})(A_1A_2^{-1}),\\
C_2A_1^{-1}&=&(C_2C_1^{-1})(C_1A_2^{-1})(A_2A_1^{-1}),\\
C_2B_1^{-1}&=&(C_2A_1^{-1})(A_1B_1^{-1})\textrm{ and }\\
B_3B_1^{-1}&=&(B_2C_1^{-1})^{T^{2}}(C_2B_1^{-1}).
\end{eqnarray*}
\begin{figure}[hbt!]
\begin{center}
\scalebox{0.45}{\includegraphics{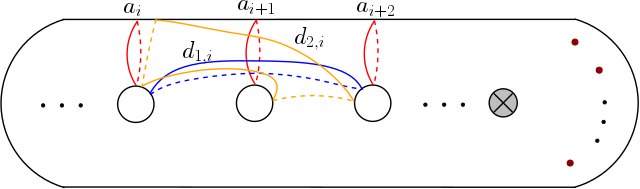}}
\caption{The curves $d_{1,i}$ and $d_{2,i}$.}
\label{curves2}
\end{center}
\end{figure}
Using the lantern relation given in the proof of \cite[Theorem 2.1]{apy2}, it can be verified that the subgroup $G$ contains the elements $A_1$, $A_2$, $A_3$, $B_i$ and $C_i$ for $1\leq i \leq r-1$ and $B_r$ (also the element $C_r$ when $g$ is even).

We now inductively show that $A_j\in G$ for $j=4, \ldots, r$. To do this, we first use the diffeomorphism
\[
(A_iB_{i+1}^{-1})(A_iC_i^{-1})(A_iC_{i+1}^{-1})(A_iB_{i+1}^{-1}),
\]
which sends the curves $(a_{i+1},a_i)$ to $(d_{2,i},a_i)$, where the curves are shown in Figure~\ref{curves2} for $i+2\leq r$.
This gives rise to the following identity:
\[
\big((A_iB_{i+1}^{-1})(A_iC_i^{-1})(A_iC_{i+1}^{-1})(A_iB_{i+1}^{-1})\big)(A_{i+1}A_i^{-1})\big((A_iB_{i+1}^{-1})(A_iC_i^{-1})(A_iC_{i+1}^{-1})(A_iB_{i+1}^{-1})\big)^{-1}=D_{2,i}A_i^{-1}.
\]
Thus, we have $D_{2,i}C_{i+1}^{-1}=(D_{2,i}A_i^{-1})(A_iC_{i+1}^{-1})$. Similarly, one can show that
\[
(C_{i+1}B_i^{-1})(C_{i+1}A_i^{-1})(C_{i+1}C_i^{-1})(C_{i+1}B_i^{-1})(d_{2,i}, c_{i+1})=(d_{1,i}, c_{i+1}),
\]
which implies that 
\[
\big((C_{i+1}B_i^{-1})(C_{i+1}A_i^{-1})(C_{i+1}C_i^{-1})(C_{i+1}B_i^{-1}) \big)(D_{2,i}C_{i+1}^{-1})\big((C_{i+1}B_i^{-1})(C_{i+1}A_i^{-1})(C_{i+1}C_i^{-1})(C_{i+1}B_i^{-1}) \big)^{-1}=D_{1,i}C_{i+1}^{-1}.
\]
We also have $D_{1,i}A_i^{-1}=(D_{1,i}C_{i+1}^{-1})(C_{i+1}A_i^{-1})$. By the lantern relation given in the proof of ~\cite[Theorem 2.1]{apy2}, we get
\[
A_{i+2}=(A_{i+1}C_i^{-1})(D_{1,i}C_{i+1}^{-1})(D_{2,i}A_i^{-1}).
\]
Therefore, since $A_1$, $A_2$ and $A_3$ belong to $G$, we inductively obtain that all $A_j\in G$ for $j=1,\ldots,r$. \\
On the other hand, since $\rho_3(a_1)=f_1$ and $\rho_4(a_i)=f_i$ for $i=2,\ldots,r$,
\[
F_1=(A_1^{-1})^{\rho_3} \in G \textrm{ and } F_i=(A_i^{-1})^{\rho_4} \in G \ \textrm{ for } \ i=2,\ldots,r.
\]
This completes the proof.
\end{proof}
Let $G$ be the subgroup of $\mod(N_{g,p})$
generated by the elements given explicitly in Lemma~\ref{lemma2} with the given conditions.
\begin{lemma}\label{lemma3}
The subgroup $G$ contains the elements $y$ and $v_{g,i}$ (and also $v_{g-1,i}$ if $g$ is even) for $i=1,2,\ldots,p$.
\begin{proof}
Let us first prove that $G$ contains a $y$-homeomorphism. By Lemma~\ref{lemma2}, the crosscap transposition 
\[
u_{r+5}=(B_1^{-1})(A_2^{-1})^{T^{4}}(\Gamma_5B_1u_{r+5})\in G
\]
since each factors on the right hand side are contained in $G$. Using the action of $T$, we get $u_i\in G$ for $i=1,\ldots,g-1$. In particular, $u_1\in G$ and it is easy to see that
\[
y=A_1u_1\in G.
\]
In order to show that $v_{g,i}\in G$ for each $i=1,2,\ldots,p$, observe that
\[
v_{r+3,p}^{T^{k}}=v_{r+3+k,p}\in G \textrm{ for }k=1,\ldots g-r-3.
\]
In particular, $v_{g,p}\in G$. Using the element $\rho_3\rho_4$, we have the following elements:
\[
v_{g,p}^{\rho_3\rho_4}=v_{g,1},
\quad v_{g,1}^{\rho_3\rho_4}=v_{g,2},\quad \ldots\quad  ,v_{g,p-2}^{\rho_3\rho_4}=v_{g, p-1},
\]
which are all contained in $G$.

In the case of even $g$, similarly, using the elements $v_{g-1,p}$ and $\rho_3\rho_4$ that are contained in $G$, we get the elements
\[
v_{g-1,p}^{\rho_3\rho_4}=v_{g-1,1}\in G,
\quad v_{g-1,1}^{\rho_3\rho_4}=v_{g-1,2}\in G,\quad \ldots\quad  ,v_{g-1,p-2}^{\rho_3\rho_4}=v_{g-1,p-1}\in G,
\]
which completes the proof.
\end{proof}
\end{lemma}
\begin{lemma}\label{lemma4}
The group $\mod_{0}(N_{g,p})$ is contained in the group $G$.
\end{lemma}
\begin{proof}
The group $G$ contains the Dehn twist elements $A_i$, $B_i$, $F_i$ and $C_j$ for $i=1,\ldots,r$ and $j=1,\ldots,r-1$ (also $C_r$ when $g$ is even) by Lemma~\ref{lemma2} and elements $y$ and $v_{g,k}$ (and also $v_{g-1,k}$ if $g$ is even) for $k=1,2,\ldots,p$ by Lemma~\ref{lemma3}. It follows from Theorem~\ref{thmmk} that it is enough to prove that $E_{l}\in G$ for $l=1,2,\ldots,p-1$. Since the reflection $\rho_4$ maps $a_1$ to $e_{p-1}$, we have
\[
(A_1^{-1})^{\rho_4}=E_{p-1}\in G.
\]
The reflection $\rho_3$ maps the curve $e_{p-1}$ to $e_1$. Then we have
\[
(E_{p-1}^{-1})^{\rho_3}=E_1\in G.
\]
Similarly, using the reflection $\rho_4$, we get
\[
(E_1^{-1})^{\rho_4}=E_{p-2}\in G.
\]
Continuing in this way, we conclude that the elements $E_{1},E_{2},$ $\ldots,E_{p-1}$ belong to $G$, which completes the proof.
\end{proof}
Now, it is time to prove the main theorems.
\begin{theorem} \label{element}
The subgroup $G$ is equal to the
mapping class group $\mod(N_{g,p})$.
\end{theorem}
\begin{proof}
By Lemma~\ref{lemma4}, the group $\mod_{0}(N_{g,p})$ is contained in the group $G$. Hence, by Lemma~\ref{lemma1}, we need to prove that $G$ is mapped surjectively onto $Sym_p$. The elements $\rho_3\rho_4 \in G$ and $\rho_1\in G$ have the images $(1,2,\ldots,p)\in Sym_p$ and $(1,2)\in Sym_p$, respectively. This finishes the proof since these elements generate the whole group $Sym_p$
\end{proof}
Consider the surface $N_{g,p}$ of genus $g\geq 14$. Since 
\[
\rho_3(\gamma_5)=\gamma_5, \quad \rho_3(b_1)=b_1,\textrm{ and } \rho_3(\alpha_{r+5})=\alpha_{r+5},
\]
with reverse orientation, the element
$\rho_3\Gamma_5B_1u_{r+5}$ is an involution. Similarly, if $g=2r+1$ or $2r+2$, it follows from $v_{p,r+3}^{\rho_1}=v_{p,r+3}^{-1}$ that the element $\rho_1v_{p,r+3}$ is an involution.

\begin{theorem}\label{invo}
For any  $g=2r+1\geq15$ or $g=2r+2\geq14$, the mapping class group $\mod(N_{g,p})$ is generated by the involutions
\[
\lbrace \rho_1, \rho_2, \rho_3, \rho_4, \rho_3\Gamma_5B_1u_{r+5}, \rho_1v_{p,r+3} \rbrace
\]
\end{theorem}
\begin{proof}
The proof follows from Theorem~\ref{element} and by the fact that $T=\rho_2\rho_1$.
\end{proof}

%%%%%%%%%%%%%%%%%%%%%%%%%%%%%%%%%%%%%%%%%%%%%%%%%%%%%%%%%%%%%%%%%

\end{document}